\documentclass[11pt]{amsart}

\usepackage{amsmath,amssymb,amsthm,graphics,amscd}
\usepackage{longtable}
\usepackage{cite,xypic}
\usepackage{color}
\usepackage{graphicx}
\usepackage{epsf}
\usepackage{fancyhdr}
\usepackage{lscape}
  \renewenvironment{thebibliography}[1]{
    \begin{oldthebibliography}{#1}
      \setlength{\parskip}{0ex}
      \setlength{\itemsep}{0ex}
  }
  {
    \end{oldthebibliography}
  }
  
\setlongtables
\begin{document}

\newcounter{rownum}
\setcounter{rownum}{0}
\newcommand{\ab}{\addtocounter{rownum}{1}\arabic{rownum}}
\newcommand{\im}{\mathrm{im}}
\newcommand{\x}{$\times$}
\newcommand{\bb}{\mathbf}
\newcommand{\Lie}{\mathrm{Lie}}
\newcommand{\Char}{\mathrm{char}}
\newcommand{\hra}{\hookrightarrow}
\newtheorem{lemma}{Lemma}[section]
\newtheorem{theorem}[lemma]{Theorem}
\newtheorem*{TG1}{Theorem G1}
\newtheorem*{CG2}{Corollary G2}
\newtheorem*{CG3}{Corollary G3}
\newtheorem*{T1}{Theorem 1}
\newtheorem*{T2}{Theorem 2}
\newtheorem*{C2}{Corollary 2}
\newtheorem*{C3}{Corollary 3}
\newtheorem*{T4}{Theorem 4}
\newtheorem*{C5}{Corollary 5}
\newtheorem*{C6}{Corollary 6}
\newtheorem*{C7}{Corollary 7}
\newtheorem*{C8}{Corollary 8}
\newtheorem*{claim}{Claim}
\newtheorem{corollary}[lemma]{Corollary}
\newtheorem{conjecture}[lemma]{Conjecture}
\newtheorem{prop}[lemma]{Proposition}
\theoremstyle{remark}
\newtheorem{remark}[lemma]{Remark}
\newtheorem{obs}[lemma]{Observation}
\theoremstyle{definition}
\newtheorem{defn}[lemma]{Definition}

  \def\hal{\unskip\nobreak\hfil\penalty50\hskip10pt\hbox{}\nobreak
  \hfill\vrule height 5pt width 6pt depth 1pt\par\vskip 2mm}

\renewcommand{\labelenumi}{(\roman{enumi})}
\newcommand{\Hom}{\mathrm{Hom}}
\newcommand{\Ext}{\mathrm{Ext}}
\newcommand{\soc}{\mathrm{Soc}}
\newcommand{\Aut}{\mathrm{Aut}}
\newcommand{\N}{\mathcal{N}}

\newenvironment{changemargin}[1]{%
  \begin{list}{}{%
    \setlength{\topsep}{0pt}%
    \setlength{\topmargin}{#1}%
    \setlength{\listparindent}{\parindent}%
    \setlength{\itemindent}{\parindent}%
    \setlength{\parsep}{\parskip}%
  }%
  \item[]}{\end{list}}

\parindent=0pt
\addtolength{\parskip}{0.5\baselineskip}

 \title[Non-$G$-cr subgroups]{Non-$G$-completely reducible subgroups of the exceptional algebraic groups}
\author{David I. Stewart}
\address{New College, Oxford\\ Oxford, UK} \email{david.stewart@new.ox.ac.uk {\text{\rm(Stewart)}}}
\pagestyle{plain}
\begin{abstract}Let $G$ be an exceptional algebraic group defined over an algebraically closed field $k$ of characteristic $p>0$ and let $H$ be a subgroup of $G$. Then following Serre we say $H$ is $G$-completely reducible or $G$-cr if, whenever $H$ is contained in a parabolic subgroup $P$ of $G$, then $H$ is in a Levi subgroup of that parabolic. Building on work of Liebeck and Seitz, we find all triples $(X,G,p)$ such that there exists a closed, connected, simple non-$G$-cr subgroup $H\leq G$ with root system $X$.\end{abstract}
\maketitle
\section{Introduction}

Let $G$ be an algebraic group defined over an algebraically closed field $k$ of characteristic $p>0$ and let $H$ be a subgroup of $G$. Then following Serre \cite{Ser98} we say $H$ is $G$-completely reducible or $G$-cr if, whenever $H$ is contained in a parabolic subgroup $P$ of $G$, then $H$ is in a Levi subgroup of that parabolic. This is a natural generalisation of the notion of a group acting completely reducibly on a module $V$: if we set $G=GL(V)$ then saying $H$ is $G$-completely reducible is precisely the same as saying that $H$ acts semisimply on $V$.

This notion is important in unifying some other pre-existing notions and results. For instance, in \cite{BMR05}, it was shown that a subgroup $H$ is $G$-cr if and only if it satisfied Richardson's notion of being strongly reductive in $G$. It also allows one to state some previous results due to Liebeck--Seitz and Liebeck--Saxl--Testerman on the subgroup structure of the exceptional algebraic groups in a particularly satisfying form:

Assume $G$ is simple of one of the five exceptional types and let $X$ be a simple root system. The result \cite[Theorem 1]{LS96} asserts a number $N(X,G)$ such that if $H$ is closed, connected and simple, with root system $X$, then $H$ is $G$-cr whenever the characteristic $p$ of $k$ is bigger than $N(X,G)$. In particular if $p$ is bigger than $7$ then they show that all closed, connected, reductive subgroups of $G$ are $G$-cr. There is some overlap in that paper with the contemporaneous work of \cite{LST96}. If $H$ is a simple subgroup of rank greater than half the rank of $G$, then [Theorem 1, {\it ibid.}] finds all conjugacy classes of simple subgroups of $G$, the proofs indicate where these conjugacy classes are $G$-completely reducible. With essentially one class of exceptions, all subgroups, including the non-$G$-cr subgroups, can be located in `nice' so-called subsystem subgroups of $G$. We shall mention these in greater detail later.

More recently, \cite{SteG2} and \cite{SteF4} find all conjugacy classes of simple subgroups of exceptional groups of types $G_2$ and $F_4$. One consequence of this is to show that the numbers $N(X,G)$ found above can be made strict. (One need only change $N(A_1,G_2)$ from $3$ to $2$.) The main purpose of this article is to make {\it all} the $N(X,G)$ strict. That is, for each of the five types of exceptional algebraic group $G$, for each prime $p=\Char\ k$ and for each simple root system $X$, we give in a table of Theorem 1 an example $H=E(X,G,p)$ of a connected, closed, simple non-$G$-cr subgroup $H$\footnote{(thus, $H$ is in some parabolic $P$, but in no Levi subgroup $L$ of $P$)} with root system $X$, precisely when this is possible. In other words we classify the triples $(X,G,p)$ where there exists a connected, closed, simple non-$G$-cr subgroup $H$ with root system $X$. Moreover, in all but one case (where $(X,G,p)=(G_2,E_7,7)$), we can locate $E(X,G,p)$ in a subsystem subgroup. 

Our main theorem can thus be viewed as the best possible improvement of the result \cite[Theorem 1]{LS96}, in the spirit of that result. Before we state our main theorem in full, we need a definition: A \emph{subsystem subgroup} of $G$ is a simple, closed, connected subgroup $Y$ which is normalised by a maximal torus $T$ of $G$. Let $\Phi$ be the root system of $G$ corresponding to a choice of Borel subgroup $B\geq T$ and for $\alpha\in\Phi$, let $U_\alpha$ denote the $T$-root subgroup corresponding to $\alpha$. Then $Y=\langle U_\alpha|\alpha\in \Phi_0\rangle$ where either $\Phi_0$ is a closed subsystem of $\Phi$ or $(\Phi,p)$ is $(B_n,2)$, $(C_n,2)$, $(F_4,2)$ or $(G_2,3)$ and $\Phi_0$ lies in the dual of a closed subsystem. The subsystem subgroups of $G$ are easily determined by the Borel--de Siebenthal algorithm. Most of our examples $H=E(X,G,p)$ are described in terms of an embedding of $H$ into a subsystem subgroup $M$. Here we describe $M$ just by giving its root system.

\begin{T1} Let $G$ be an exceptional algebraic group defined over an algebraically closed field $k$ of characteristic $p>0$. Suppose there exists a non-$G$-cr closed, connected, simple subgroup $H$ of $G$ with root system $X$. Then $(X,G,p)$ has an entry in Table \ref{t1}. 

Conversely, for each $(X,G,p)$ given in Table \ref{t1}, the last column guarantees an example of a closed, connected, simple, non-$G$-cr subgroup $E(X,G,p)$ with root system $X$.\end{T1}
\begin{table}\label{t1}\begin{tabular}{|c|c|c|c|}\hline
$G$ & $X$ & $p$ & Example $E=E(X,G,p)$\\\hline\hline
$G_2$ & $A_1$ & $2$ & $E\hookrightarrow A_1\tilde A_1$; $x\mapsto (x,x)$\\\hline\hline
$F_4$ & $A_1$ & $2$ & $E\hookrightarrow A_1^2$; $x\mapsto (x,x)$\\
&&$3$ & $E\hookrightarrow A_2^2$; $(V_3,V_3)\downarrow E=(2,2)$\\\hline
&$A_2$ & $3$ & $E\hookrightarrow A_2^2$; $x\mapsto (x,x)$\\\hline
&$B_2$ & $2$ & $E\leq D_3$\\\hline
&$G_2$ & $2$ & $E\leq D_4$\\\hline
&$B_3$ & $2$ & $E\leq D_4$\\\hline\hline
$E_6$ & $A_1$ & $2$ & $E\hookrightarrow A_1^2$; $x\mapsto (x,x)$\\
& & $3$ & $E\hookrightarrow A_2^2$; $(V_3,V_3)\downarrow E=(2,2)$\\
& & $5$ & $E\hookrightarrow D_5$; $V_{10}\downarrow E=T(8)$\\\hline
& $A_2$ & $2$ & $E\hookrightarrow A_5$; $V_6\downarrow E=V(20)=10^{[1]}/01$\\
& & $3$ & $E\hookrightarrow A_2^3$; $x\mapsto (x,x,x)$\\\hline
& $B_2$ & $2$ & $E\hookrightarrow A_4$; $V_5\downarrow E=V(10)=10/00$\\\hline
& $G_2$ & $2$ & $E\hookrightarrow C_4$; $V_8\downarrow E=T(10)$\\\hline
& $B_3$ & $2$ & $E\hookrightarrow C_4$; $V_8\downarrow E=T(100)$\\\hline\hline
$E_7$ & $E\leq E_6$ & $2,3,5$ & each of the subgroups of $E_6$ above\\\hline
& $A_1$ & $7$ & $E\leq A_7$; $V_8\downarrow E=W(7)=1^{[1]}/5$\\\hline
& $G_2$ & $7$ & $E$ in an $E_6$-parabolic of $G$ *\\\hline
& $C_4$ & $2$ & $E\hookrightarrow A_7$; $V_8\downarrow E=L(1000)$\\\hline
& $D_4$ & $2$ & $E\leq C_4$ above\\\hline\hline
$E_8$ & $E\leq E_7$ & $2,3,5,7$ & each of the subgroups of $E_7$ above\\\hline
& $B_2$ & $5$ & $E\leq D_8$; $V_{16}\downarrow E=T(20)=00/20/00$\\\hline
& $A_3$ & $2$ & $E\leq D_8$; $V_{16}\downarrow E=T(101)=000/101/000$\\\hline
& $C_3$ & $3$ & $E\leq D_8$; $V_{16}\downarrow E=000/010/000+000$?\\\hline
& $B_4$ & $2$ & $E\leq A_8$; $V_9\downarrow E=1000/0000$\\\hline
\end{tabular}\caption{Simple non-$G$-cr subgroups of type $X$ in the exceptional groups}\end{table}

In particular we can improve on \cite[Theorem 1]{LS96}. In the table in Corollary 2 we have struck out the primes which were used in the hypotheses in [loc. cit.]. This is done partly to show where we have made improvements but mainly to facilitate reading the proof of the first part of Theorem 1.

\begin{C2}Let $G$ be an exceptional algebraic group over a field $k$ of characteristic $p$. Let $X$ be a simple root system and let $N(X,G)$ be a list of primes defined by the table below. Suppose $H$ is a closed, connected, reductive subgroup of $G$ with root system having simple components $X_1,\dots,X_n$. Then if $p\not\in\bigcup_i N(X_i,G)$, $H$ is $G$-cr.
\begin{center}\begin{tabular}{r|rrrrr}
&$G=E_8$&$E_7$&$E_6$ & $F_4$ & $G_2$\\
\hline
$X=A_1$ & $\leq 7$ & $\leq 7$ & $\leq 5$ & $\leq 3$ & $\not3\ 2$\\
$A_2$ & $\not 5\ 3\ 2$ & $\not 5\ 3\ 2$ & $3\ 2$ & $3 \not 2$\\
$B_2$ & $5 \not 3\ 2$ & $\not 3\ 2$ & $\not 3\ 2$ & $2$\\
$G_2$ & $7\not 5\ 3\ 2$ & $7\not 5\ 3\ 2$ & $\not 3\ 2$ & $ 2$\\
$A_3$ & $2$ & $\not 2$ & $\not 2$ & \\
$B_3$ & $2$ & $2$ & $2$ & $2$\\
$C_3$ & $3\ \not 2$ & $\not 2$ & $\not 2$ & $\not 2$\\
$B_4$ & $2$ & $\not2$ & $\not2$\\
$C_4, D_4$ & $2$ & $2$ & $\not 2$\end{tabular}
\end{center}\end{C2}

Using the above description of $N(X,G)$ one also gets generalisations to each of the other results \cite[Theorems 2--8]{LS96}, by replacing the hypothesis `$p>N(X,G)$' by `$p\not\in N(X,G)$'. 

\section{Notation}
When discussing roots or weights, we use the Bourbaki conventions \cite[VI. Planches I-IX]{Bourb82}. We use a lot of representation theory for algebraic groups whose notation we have taken largely consistent with \cite{Jan03}. For an algebraic group $G$, recall that a $G$-module is a comodule for the Hopf algebra $k[G]$; in particular every $G$-module is a $kG$-module. Let $B$ be a Borel subgroup of a reductive algebraic group $G$, containing a maximal torus $T$ of $G$. Recall that for each dominant weight $\lambda\in X^+(T)$ for $G$, the space $H^0(\lambda):=H^0(G/B,\lambda)=\mathrm{Ind}_B^G(\lambda)$ is a $G$-module with highest weight $\lambda$ and with socle $\mathrm{Soc}_G H^0(\lambda)=L(\lambda)$, the irreducible $G$-module of highest weight $\lambda$. The Weyl module of highest weight $\lambda$ is $V(\lambda)\cong H^0(-w_0\lambda)^*$ where $w_0$ is the longest element in the Weyl group. We identify $X(T)$ with $\mathbb Z^r$ for $r$ the rank of $G$ and for $\lambda\in X(T)^+\cong\mathbb Z_{\geq 0}^r\leq X(T)$, write $\lambda=(a_1,a_2,\dots,a_r)=a_1\omega_1+\dots+a_r\omega_r$ where $\omega_i$ are the fundamental domninant weights; a $\mathbb Z_{\geq 0}$-basis of $X(T)^+$. Put also $L(\lambda)=L(a_1,a_2,\dots,a_r)$.  When $0\leq a_i<p$ for all $i$, we say that $\lambda$ is a restricted weight and we write $\lambda \in X_1(T)$. Recall that any module $V$ has a Frobenius twist $V^{[n]}$ induced by raising entries of matrices in $GL(V)$ to the $p^n$th power. Steinberg's tensor product theorem states that $L(\lambda)=L(\lambda_0)\otimes L(\lambda_1)^{[1]}\otimes\dots\otimes L(\lambda_n)^{[n]}$ where $\lambda_i\in X_1(T)$ and $\lambda=\lambda_0+p\lambda_1+\dots+p^n\lambda_n$ is the $p$-adic expansion of $\lambda\in\mathbb Z_+^r$. We refer to $\lambda_0$ as the restricted part of $\lambda$.

The right derived functors of $\Hom(V,*)$ are denoted by $\Ext_G^i(V,*)$ and when $V=k$, the trivial $G$-module, we have the identity $\Ext_G^i(k,*)=H^i(G,*)$ giving the Hochschild cohomology groups. 

We recall some standard modules; when $G$ is classical, there is a `natural module' which we refer to by $V_\text{nat}$; or $V_m$ where $m$ is the dimension of $V_\text{nat}$. It is always the Weyl module $V(\omega_1)$, which is irreducible unless $p=2$ and $G$ is of type $B_n$; in the latter case it has a $1$-dimensional radical. Certain properties of these modules is described in \cite[8.21]{Jan03}. Of importance to us is the fact that when $G=SL_n$, $\bigwedge^r(L(\omega_1))=L(\omega_r)$ for $r\leq n-1$. We use this fact without further reference.

Recall that $F_4$ has a $26$-dimensional Weyl module which we denote `$V_{26}$'. When $p\neq 3$, $V_{26}$ is the irreducible representation of high weight $0001=\omega_4$. When $p=3$, $V_{26}$ has a one-dimensional radical, with a $25$-dimensional irreducible quotient of high weight $0001$. $E_6$ (resp. $E_7$, $E_8$) has a module of dimension $27$ (resp. $56$, $248$) of high weight $\omega_1$ (resp. $\omega_7$, $\omega_8$) which is irreducible in all characteristics. We refer to this module as $V_{27}$ (resp. $V_{56}$, $\Lie(E_8)$).

We will often want to consider restrictions of simple $G$-modules to reductive subgroups $H$ of $G$. Where we write $V_1|V_2|\dots |V_n$ we list the composition factors $V_i$ of an $H$-module. For a direct sum of $H$-modules, we write $V_1+V_2$. Where a module is uniserial, we will write $V_1/\dots/V_n$ to indicate the socle and radical series: here the head is $V_1$ and the socle $V_n$. On rare occasions we use $V/W$ to indicate a quotient. It will be clear from the context which is being discussed. 

Recall also the notion of a tilting module as one having a filtration by modules $V(\mu)$ for various $\mu$ and also a filtration by modules $H^0(\mu)$ for various $\mu$ (equiv. dual Weyl modules). Let us record in a lemma some key properties of tilting modules which we use:\begin{lemma}\label{tilting}\begin{enumerate}\item For each $\lambda\in X(T)^+$ there is a unique indecomposable tilting module $T(\lambda)$ of high weight $\lambda$; 
\item A direct summand of a tilting module is a tilting module;
\item The tensor product of two tilting modules is a tilting module;
\item $\Ext^1_G(T(\lambda),T(\mu))=0$; in particular $H^1(G,T(\lambda))=0$.
\end{enumerate}\end{lemma}\begin{proof} For (i), see \cite[1.1(i)]{Don93}; (ii) is clear by projecting a filtration to a direct summand and using the fact that Weyl modules and induced modules are indecomposable; (iii) is \cite[1.2]{Don93}; (iv) follows from \cite[II.4.13 (2)]{Jan03}.\end{proof}
As we are considering very low weight representations in general, it is possible to spot that a module is a $T(\lambda)$; for instance when $p=2$, the natural Weyl module for $B_n$ has a $1$-dimensional radical, so its structure is  $W(\lambda_1)=L(\lambda_1)/k$. It is then the case that giving the Loewy series for a module $k/L(\lambda_1)/k$ uniquely characterises it as a tilting module $T(\lambda_1)$.

Recall that a parabolic subgroup $P$ of $G$ has a Levi decomposition, \[1\to Q\to P\to L\to 1\] where $Q$ is the unipotent radical of the $P$. Recall also $L=L'Z(L)$ with $L'$ being semisimple.

\section{Outline}Theorem 1 has two facets. The first proves that if $p\not\in N(X,G)$ for $N(X,G)$ as defined in Corollary 2, then $X$ is $G$-cr. The second proves the existence of the examples given in Table \ref{t1} and proves that they are non-$G$-cr. 

The proof of the first part runs along the same lines as that of \cite[Theorem 1]{LS96}:  Assume $H$ is a closed, connected, simple non-$G$-cr subgroup of $G$. Then $H$ is a subgroup of $P=LQ$; let $\bar H$ be its image in $L'$. Almost all the time, $H\cap Q=\{1\}$ as group-schemes and so we have $HQ=\bar HQ$ and $H$ is a complement to $Q$ in $\bar HQ$. Then the possibilities for $H$ are parameterised by $H^1(\bar H,Q)$; in fact, in any case,  the possibilities for $H$ are parameterised by $H^1(\bar H,Q^{[1]})$. This is the content of \cite[Lemma 3.6.1]{SteRes}.

From \cite{ABS90}, $Q$ has a filtration $Q=Q_1\geq Q_2\geq Q_3\dots$ with successive quotients being known (usually semisimple) $L$-modules. So if we have $H^1(\bar H,(Q_i/Q_{i+1})^{[1]})=0$ for each $i$, then (by \ref{comp}(ii)) $H^1(\bar H,Q^{[1]})=0$ and $H$ is conjugate to $\bar H$.

Now, for an exceptional algebraic group $G$ over $k$ of characteristic $p$ and a simple root system $X$ we consider possible embeddings $\bar H\leq L'$ where $\bar H$ is an $L'$-irreducible subgroup (which can be determined using \ref{girclassical} and/or by working down through \ref{excepmax}). The composition factors $V$ of the restrictions of the $L$-modules $Q_i/Q_{i+1}$ are investigated, and then conditions for the vanishing of $H^1(\bar H,V)$ found, for all relevant $V$. (Usually the dimensions of the composition factors are too small to admit non-vanishing of $H^1(\bar H,V)$.) 

With essentially one exception, one can reduce to the case where $V$ is of the form $L(\lambda)\otimes L(\mu)^{[1]}$ with $L(\lambda)$ non-trivial and restricted. There are any number of computer programs one can use to calculate the values of $H^1(X,V)$ where $\mu$ is $0$.\footnote{We use the data on Frank L\"ubeck's website which accompanies \cite{Lub01}.} Since the possible dimension of $V$ is limited to a subset of roots of $G$, this process is finite. 

For the proof of the second part of Theorem 1, we must show that for each of the remaining cases (where some composition factor $V$ of $Q^{[1]}$ has $H^1(\bar H,V)\neq 0$), we exhibit a non-$G$-cr subgroup $H$ with the required root system over the required characteristic. In almost all cases we can give an example in a classical subgroup of $G$. Here is it easy to see when it is in a parabolic subgroup using \ref{girclassical}. In two cases this is not possible, yet we can assert the existence of such a group using a cohomological argument.

\section{Preliminaries}
One needs to be careful about the notion of complements in semidirect products of algebraic groups. These are treated systematically in \cite{McN10}. We recall some of the main facts.

\begin{defn}[cf. {\cite[4.3.1]{McN10}}]\label{compdef} Let $G=H\ltimes Q$ be a semidirect product of algebraic groups as in \cite[I.2.6]{Jan03}.

A closed subgroup $H'$ of $G$ is a \emph{complement} to $Q$ if it satisfies the following equivalent conditions:
\begin{enumerate}\item Multiplication is an isomorphism $H'\ltimes Q\to G$.
\item $\pi_{H'}:H'\to H$ is an isomorphism of algebraic groups
\item As group-schemes, $H'Q=G$ and $H'\cap Q=\{1\}$.
\item For the groups of $k$-points, one has $H'(k)Q(k)=G(k)$, $H'(k)\cap Q(k)=\{1\}$ and $\Lie(H')\cap \Lie(Q)=0$.\end{enumerate}\end{defn}
\begin{remark} See \cite[\S3.2]{SteRes} for a discussion. Note that \cite{LS96} uses item (iv) above as its definition of a complement, without the last condition on Lie algebras.\end{remark}

\begin{defn}A rational map $\gamma:H\to Q$ is a \emph{$1$-cocycle} if $\gamma(nm)=\gamma(n)^m\gamma(m)$ for each $n,m\in N(k)$. We write $Z^1(H,Q)$ for the set of $1$-cocycles.

We say $\gamma\sim\delta$ if there is an element $q\in Q(k)$ with $q^{-h}\gamma(h)q=\delta(h)$ for each $h\in H(k)$. We write $H^1(H,Q)$ for the set of equivalence classes of $1$-cocycles $Z^1(H,Q)/\sim$.
\end{defn}

We recall some results from \cite{SteRes}.

\begin{lemma}\label{comp}\begin{enumerate}\item The set of $1$-cocycles $Z(H,Q)$ is in bijection with the set of complements to $Q$ in $HQ$. Two cocycles are equivalent if the corresponding complements are conjugate by an element of $H(k)$.
\item Suppose $H$ is a closed, connected, reductive subgroup of a parabolic subgroup $P=LQ$ of $G$ and denote by $\bar H$ the subgroup of $L$ given by the image of $H$ under the quotient map $\pi:P\to L$.

Then as abstract groups $H(k)$ is a complement to $Q(k)$ in $\bar H(k)Q(k)$; and either (1) $H$ is a complement to $Q$ in $\bar HQ$; or (2) \begin{enumerate}\item $p=2$;
\item There exists a component $SO_{2n+1}$ of the semisimple group $H/Z(H)^\circ$;
\item the image of this component in $\bar H/Z(\bar H)^\circ$ is isomorphic to $Sp_{2n}$; and
\item the natural module for $Sp_{2n}$ appears in a filtration of $Q$ by $\bar H$-modules.\end{enumerate} In case (2), $H$ corresponds to a cocycle $\gamma\in Z^1(\bar H,Q^{[1]})$ such that $[\gamma]$ has no preimage in $H^1(\bar H,Q)$ under the inclusion $H^1(\bar H,Q)\to H^1(\bar H,Q^{[1]})$.

Thus there is a bijection between the set of conjugacy classes of closed, connected, reductive  subgroups $H$ of $\bar HQ$ and the set $H^1(\bar H,Q^{[1]})$.
\item In a filtration of a unipotent algebraic $H$-group $Q$ by $H$-modules (such as that given by \ref{abs}) if each composition factor $V$ satisfies $H^1(H,V)=0$ then $H^1(H,Q)=0$.
\end{enumerate}\end{lemma}
\begin{proof}(i) is \cite[Lemma 3.2.2]{SteRes}; (ii) is \cite[Lemma 3.6.1]{SteRes}.
For (iii), such a filtration is `sectioned' in the sense of \cite[Definition 3.2.7]{SteRes} using \cite[Lemma 3.2.8]{SteRes}. Now one uses  the exact sequence of non-abelian cohomology in \cite[2.1(i)]{SteRes} inductively. (See the discussion in \cite[\S3.2]{SteRes} on the validity of this sequence for rational cohomology.)\end{proof}

In almost all cases the cohomology group $H^1(G,V)$ for a semisimple algebraic group $G$ satisfies $H^1(G,V)\cong H^1(G,V^{[1]})$. This fact allows us to reduce our considerations to simple modules with non-trivial restricted parts.
\begin{lemma}\label{cnrep}Let $G$ be a simple algebraic group and $V$ a simple $G$-module. Then $H^1(G,V)\cong H^1(G,V^{[1]})$ unless $G$ is $Sp_{2n}$ and $V$ is its $2n$-dimensional natural module.

Moreover $H^1(G,V^{[1]})$ is isomorphic to its generic cohomology $H^1_\text{gen}(G,V)$.\end{lemma}
\begin{proof}See \cite[II.12.2, Remark]{Jan03} and \cite[7.1]{CPSV77}. \end{proof}

There are many papers finding the values $\Ext^n_H(L,M)$ with $H$ of low rank and $L,M$ simple. Taking $L=k$, the trivial module, one gets the following result, where we have included more data than necessary for our purposes for completion's sake.

\begin{lemma}\label{h1fora1a2b2a3} Let $V$ be a simple module for a simple algebraic group $H$ where $H$ is one of $SL_2$, $SL_3$, $Sp_4$ over an algebraically closed field of any characteristic $p$; $G_2$ for $p=2,3$ or $p\geq 13$; or $SL_4$, $Sp_6$ or $Sp_8$ when $p=2$. Then $H^1(H,V)$ is at most one-dimensional, and is non-zero if and only if $V$ is a Frobenius twist of one of the modules in the following table. 

In the table we also give some useful dimension data, often in only specific characteristic. 
{\footnotesize 
\begin{table}[ht]
\parbox[t]{.45\linewidth}{
\centering
\begin{tabular}[t]{|c|c|l|c|}\hline $H$ & $p$ & $L$ & dim $L$\\\hline\hline
$SL_2$ & any & $L(p-2)\otimes L(1)^{[1]}$ & $2p-2$\\\hline
$SL_3$ & $p\geq 3$ & $L(p-2,p-2)$ & $(p-1)^3-1$\\
&& $L(1,p-2)\otimes L(1,0)^{[1]}$ & $54$ for $p=5$\\
&& $L(p-2,1)\otimes L(0,1)^{[1]}$ & $54$ for $p=5$\\
& $p=2$ & $L(1,0)\otimes L(1,0)^{[1]}$ & $9$\\
&& $L(0,1)\otimes L(0,1)^{[1]}$ & $9$\\\hline
$Sp_4$ & $p\geq 5$ & $L(0,p-3)$ & \\
&$p\geq 3$ & $L(2,p-2)\otimes L(0,1)^{[1]}$ & $125$ for $p=3$\\
&&$L(p-2,1)\otimes L(1,0)^{[1]}$ & $\geq 64$\\
& $p=2$ & $L(1,0)^{[1]}$ & $4$\\
& & $L(0,1)$ & $4$\\\hline
$G_2$ & $p\geq 13$ & $L(p-5,0)$&\\
& & $L(p-2,1)\otimes L(1,0)^{[1]}$&\\
& & $L(4,p-4)\otimes L(1,0)^{[1]}$&\\
& & $L(3,p-2)$&\\
& & $L(3,p-2)\otimes L(0,1)^{[1]}$&\\
& $p=3$ & $L(1,1)$ & $49$\\
& & $L(0,1)\otimes L(1,0)^{[1]}$ & $49$\\
& $p=2$ & $L(1,0)$ & $6$\\
&& $L(0,1)\otimes L(1,0)^{[1]}$ & $84$\\\hline\end{tabular}
}
\hfill
\parbox[t]{.45\linewidth}{
\centering
\begin{tabular}[t]{|c|c|l|c|}\hline $H$ & $p$ & $L$ & dim $L$\\\hline\hline
$SL_4$ & $p=2$ & $L(1,0,1)$ & $14$\\
 && $L(0,1,0)\otimes L(1,0,0)^{[1]}$ & $24$\\
  && $L(0,1,0)\otimes L(0,0,1)^{[1]}$ & $24$\\
 && $L(1,0,1)\otimes L(0,1,0)^{[1]}$ & $84$\\\hline
 $Sp_6$ & $p=2$ & $L(1,0,0)^{[1]}$ & $6$\\
   &  & $L(1,0,1)$ & $48$\\
   &  & $L(0,1,0)\otimes L(1,0,0)^{[1]}$ & $84$\\\hline
$Sp_{8}$ & $p=2$ & $L(1,0,0,0)^{[1]}$ & $8$\\
&& $L(0,1,0,0)$ & $26$\\
&& $L(1,0,1,0)$ & $246$\\
&& $L(1,0,1,0)\otimes L(0,1,0,0)$ & $6396$\\
&& $L(1,0,1,0)\otimes L(0,1,0,0)$ & $6396$\\
&& $L(0,1,0,1)$ & $416$\\\hline
\end{tabular}
}
\end{table}

}\end{lemma}
\begin{proof}These are special cases from \cite{Cli79}  for $SL_2$, \cite[4.2.2]{Yeh82}  for $SL_3$, \cite{Ye90} for $Sp_4$, $p\geq 3$, \cite{LY93} for $G_2$ ($p\geq 13$), \cite[Proposition 2.2]{Sin94} for $Sp_4$ ($p=2$), \cite[Proposition 3.4]{Sin94} for $G_2$ ($p=3$), \cite[II.\S2.1.6, II.\S2.2.4, II.\S3.3.6, III.\S2.2.4]{DS96} for $SL_4$, $Sp_6$, $G_2$  and $Sp_8$, respectively, when $p=2$.\end{proof}

\begin{lemma}\label{h1forg2}Let $G=G_2$ over a field of characteristic $5$ and let $L$ be a simple module for $G$ with $H^1(G,L)\neq 0$. Then $\dim L>56$.
\end{lemma}
\begin{proof}One reduces to the case where the restricted part of $L$ is non-trivial using \ref{cnrep} so we may assume $L=M$. Start with the case that $M$ is restricted. One can use the data from \cite{Lub01} to establish that all Weyl modules of dimension less than $97$ are irreducible. But then $H^1(G,L(\lambda))\cong H^0(G,H^0(\lambda)/\soc_G(H^0(\lambda)))=0$.

If $M$ is not restricted, then it is $M=M_1\otimes M_2^{[1]}$ for $M_1$ restricted and $M_2$ non-trivial. The lowest dimension $M_1$ and $M_2$ can have is $7$, the next is $14$, but $14\times 7>56$, so we conclude $M_1=L(1,0)$ and $M_2=L(1,0)^{[r]}$. Now by \cite[1.15]{LS96} (or the linkage principle), one gets $H^1(G,M)=0$.\end{proof}

The next lemma is useful for establishing $L'$-irreducible embeddings $\bar H\leq L'$ when $L'$ and also for deciding when a subgroup $H$ is in a parabolic of a classical subgroup $M$ of $G$.

\begin{lemma}[{\cite[p32-33]{LS96}}] \label{girclassical}Let $G$ be a simple algebraic group of classical type, with natural module 
$V = V_G(\lambda_1)$, and let $H$ be a $G$-irreducible subgroup of $G$. 
\begin{enumerate}\item If $G = A_n$, then $H$ acts irreducibly on $V$ 
\item If $G = B_n$, $C_n$, or $D_n$ with $p\neq 2$, then $V\downarrow H = V_1 \perp\dots\perp V_k$ with the $V_i$ all non-degenerate, irreducible, and inequivalent as $X$-modules.
\item If $G = D_n$ and $p = 2$, then $V\downarrow H = V_1\perp\dots\perp V_k$ with the $V_i$ all non-degenerate, $V_2\downarrow H,\dots,V_k\downarrow H$, irreducible and inequivalent, and if $V_1\neq 0$, $H$ acting on $V_1$ as a $B_{m-1}$-irreducible subgroup where $\dim V_1 = 2m$.\end{enumerate}\end{lemma}

On a couple of occasions we need to know the reductive maximal subgroups of $E_6$ and $E_7$.

\begin{lemma}[{c.f. \cite[Theorem 1]{LS04}}]\label{excepmax}Let $G$ be an exceptional group not of type $E_8$ and let $M$ be a closed, connected, reductive maximal subgroup of $G$ without factors of $A_1$ or connected centre. Then $M$ is in the following list

\begin{center}{\small\begin{tabular}{c|l|l}$G$ & Subsystem $M$ & Non-subsystem $M$\\\hline
$G_2$ & $A_2$, $\tilde A_2$ $(p=3)$&\\
$F_4$ & $B_4$, $C_4 (p=2)$, $D_4$, $\tilde D_4$ $(p=2)$, $A_2\tilde A_2$&$G_2$ $(p=7)$\\
$E_6$ & $A_2^3$ & $A_2$ $(p\neq 2,3)$, $G_2$ $(p\neq 7)$, \\ & &$C_4$ $(p\neq 2)$, $F_4$, $A_2G_2$.\\
$E_7$ & $A_7$, $A_2A_5$ & $A_2$ $(p\geq 5)$, $G_2C_3$\end{tabular}}\end{center}\end{lemma}

\section{Proof of Theorem 1}
In \cite{SteG2} and \cite{SteF4} we find all semisimple non-$G$-cr subgroups of $G$ where $G$ is $G_2$ and $F_4$ respectively. So the result follows for these cases. It remains to deal with the cases $G=E_6$, $E_7$ and $E_8$. We start by honing the Liebeck and Seitz result to show that if $H$ is a closed, connected, simple subgroup of $G$ with root system $X$ and $p$ is not in our list $N(X,G)$ then $H$ is $G$-cr. Then we check that the examples given in Table \ref{t1} are indeed non-$G$-cr.

A filtration for unipotent radicals of parabolics by $L$-modules is given in \cite{ABS90}; to find the isomorphism types of the composition factors is a simple calculation using the root system of $G$. Summarising the results for our situation, we get:

\begin{lemma}[{\cite[3.1]{LS96}}]\label{abs}Let $G=E_6,\ E_7$ or $E_8$ and let $P=LQ$ be a parabolic subgroup of $G$. The $L'$-composition factors within $Q$ have the structure of high weight modules for $L'$. If $L_0$ is a simple factor of $L'$, then the possible high weights $\lambda$ of non-trivial $L_0$-composition factors and their dimensions are as follows:
\begin{enumerate}\item $L_0=A_n$: $\lambda=\lambda_j$ or $\lambda_{n+1-j}$ ($j=1,2,3$), dimensions $\left(\begin{array}{c}n+1\\j\end{array}\right)$;
\item $L_0=D_n$: $\lambda=\lambda_1,\ \lambda_{n-1}$ or $\lambda_n$, dimensions $2n,\ 2^{n-1}$ and $2^{n-1}$ resp.;
\item $L_0=E_6$: $\lambda=\lambda_1$ or $\lambda_6$, dimension $27$ each;
\item $L_0=E_7$: $\lambda=\lambda_7$, dimension $56$.\end{enumerate}\end{lemma}

\begin{corollary}\label{dimcor}With the hypotheses of the lemma, let $V$ be an $L'$-composition factor of $Q$ and suppose $L'$ does not contain a component of type $A_1$. Then either $\dim V\leq 60$ or $G=E_8$, $L'=D_7$ and $V$ is a spin module for $L'$ of dimension $64$.

If $G=E_7$, $\dim V\leq 35$; if $G=E_6$, $\dim V\leq 20$.\end{corollary}
\begin{proof}
If $L'$ is itself simple, this follows from the lemma. Also, if $G=E_6$ or $E_7$ then the number of positive roots is less than $56$, so the result is clear. So we may assume $G=E_8$. The possibilities for $L$ are $A_2A_2$, $A_2A_3$, $A_2A_4$, $A_3A_3$, $A_3A_4$, $A_2D_4$ and $A_2D_5$. Since $V$ is simple, it must be a tensor product of simple modules for the two factors, with the simple modules occurring in the lemma. One checks that the highest dimension possible for this is when $L=A_3A_4$, $V=L(\lambda_2)\otimes L(\lambda_2)$ with $\dim V=6\times 10=60$.

For the second part, if $G=E_7$ and $L'$ is simple this follows from Lemma \ref{abs}, the largest case occurring when $L'=A_6$. If $L'$ is not simple, then it is $A_4A_2$, $A_3A_2$ or $A_2A_2$. Then the largest possible dimension comes from the first option and is at most $10\times 3=30\leq 35$-dimensional.
\end{proof}

\subsection*{$p\not\in N(X,G)$ implies that $H$ is $G$-cr}

Since we are building on \cite[Theorem 1]{LS96}, we need only deal with the struck out numbers in the table in Corollary 2.

{\it Proof of the first statement of Theorem 1:}

Looking for a contradiction, we will assume $H$ is non-$G$-cr; then we can make the following assumption, using \ref{comp}:
\begin{quote}\it We have $H\leq P=LQ$ with $\bar H$ being $L$-ir, and either (i) $H$ is a complement to $Q$ in $\bar HQ$ and there exists a composition factor $V$ of $Q$ with $H^1(\bar H,V)\neq 0$; or (ii) $p=2$, $H=SO_{2n}$, $\bar H=Sp_{2n}$ and $V=L(\omega_1)$ appears as a composition factor of $Q$.\end{quote}

The cases to consider are \begin{align*}(X,G,p)\in\{&(B_2,\bullet,3),\ (G_2,\bullet,5),\ (G_2,E_6,3),\ (A_2,\bullet,5),\ (A_3,E_6,2),\ \\ &(A_3,E_7,2),\ (B_4,E_6,2),\ (B_4,E_7,2),\ (D_4,E_6,2),\ \\
& (C_3,\bullet,2),\ (C_4,\bullet, 2)\},\end{align*} where $\bullet$ can be replaced by $E_6$, $E_7$ or $E_8$.

By Corollary \ref{dimcor} the largest possibility for the dimension of $V$ occurs when $G=E_8$, $L'=D_7$ and $V$ has dimension $64$. By \ref{h1fora1a2b2a3}, there is no such $V$ when $H=G_2$ and $p=5$. This rules out $(G_2,\bullet,5)$. 

Suppose $H$ is of type $B_2$ and $p=3$. Since $\bar H$ is $D_7$-irreducible, it must have act on the natural module $V_{14}$ for $L'$ as specified in \ref{girclassical}. Checking \cite{Lub01}, one finds the simple untwisted representations of dimension no more than $14$ are $L(0,1)$, $L(1,0)$, $L(0,2)$, $L(2,0)$ with dimensions $4$, $5$, $10$ and $14$, respectively. But $L(0,1)$ is the natural representation for $Sp_4$, thus carries a symplectic structure, which cannot be non-degenerate. Hence $V_{14}\downarrow \bar H=L(2,0)$; moreover, as $L(2,0)$ is an irreducible Weyl module when $p=3$, the embedding $\bar H\hookrightarrow L'$ can be seen as the reduction mod $p$ of an embedding $\bar H_\mathbb Z \hookrightarrow L'_\mathbb Z$. Now \cite[Proposition 2.12]{LS96} gives that $V_\mathbb Z\downarrow \bar H_\mathbb Z$ is the irreducible Weyl module $V(1,3)$. Using \cite{Lub01} one can calculate the composition factors of a reduction mod $3$ of this module; one sees that $V\downarrow \bar H$ has composition factors $L(1,3)|L(2,1)|L(0,1)$. Since none of these modules appears in \ref{h1fora1a2b2a3}, this rules out $(X,G,p)=(B_2,\bullet,3)$.

By \ref{dimcor} the largest possibility for the dimension $V$ when $G=E_7$ is $35$; when $G=E_6$ it is $16$.

Then dimension considerations using \ref{h1fora1a2b2a3} and \ref{h1forg2} also rule out $(X,G,p)=(A_2,E_7,5)$ and $(G_2,E_6,3)$, respectively. 

For $(A_2,E_8,5)$, the fact that $V$ has dimension at least $54$ forces $L'=E_7$, $D_7$ or $A_7$ but simple $E_7$- and $D_7$-modules are self-dual, so the possibilities for $V$ coming from \ref{h1fora1a2b2a3} are discounted as they are not self-dual. Thus we may assume that $L'=A_7$ and $V=L(\omega_3)=\bigwedge^3(L(\omega_1))$. Since $\bar H$ is $L'$-ir, $L'$ must act irreducibly on the natural 8-dimensional module $V_8$ for $L'$. A check of \cite{Lub01} forces $V_8|L'=L(1,1)$. But $\bigwedge^3 L(1,1)$ has highest weights $(2,2)$ and $(0,3)$. But the weights appearing in \ref{h1fora1a2b2a3} are all higher than these (in the dominance order). This rules out $(A_2,E_8,5)$.

Consider next the case $(X,G,p)=(A_3,E_6,2)$. By \ref{dimcor} we have $\dim V\leq 20$ so \ref{h1fora1a2b2a3} shows that $V$ must be $14$-dimensional; this forces $L'=D_5$ or $A_5$. Examining low dimensional representations for $A_3$, it is easy to see using \ref{girclassical} that there is no $D_5$-irreducible embedding $\bar H\hookrightarrow D_5$, so we must have $\bar H\hookrightarrow L'=A_5$ by $V_6|\bar H=L(0,1,0)$. Here, $Q$ has factors $L(\lambda_3)=\bigwedge^3(V_6)$ and a trivial module. Now $L(0,1,0)$ has weights $\pm (0,1,0), \pm (1,0,-1), \pm (1,-1,1)$, so $\bigwedge^3L(0,1,0)$ has dominant weights $(0,0,2),\ (2,0,0)$ and $(0,1,0)$. These do not appear in \ref{h1fora1a2b2a3}. Thus $H^1(\bar H,\bigwedge^3 L(0,1,0))=0$ and this case is ruled out.

Let $(X,G,p)=(A_3,E_7,2)$. Again $V$ is at least $14$-dimensional. So $L'=A_5, A_6, D_5, D_6$ or $E_6$. Using \ref{girclassical} and \ref{excepmax} for $L'=D_5$ and $E_6$ respectively, one finds there are no $L'$-irreducible subgroups of type $A_3$. Thus $L'$ is $A_5$ or $D_6$; a similar analysis to the case $(A_3,E_6,2)$ rules out the former as an option. So we have $\bar H\leq A_3^2\leq L'=D_6$ by $V_{12}\downarrow \bar H=L(0,1,0) + L(0,1,0)^{[r]}$. Now $Q$ has $L$'-composition factors $k$ and $L(\omega_6)$, a spin module. We wish to calculate $L(\omega_6)\downarrow \bar H$. Since $\bar H\leq A_3^2$ it is instructive to work out $L(\omega_6)$ restricted to one of these factors. Using \cite[2.6 and 2.7]{LS96} this is $L(1,0,0)^4+L(0,0,1)^4$. Thus we must have $L(\omega_6)\downarrow A_3^2=(L(1,0,0),L(1,0,0))+(L(0,0,1),L(0,0,1))$ so that $L(\omega_6)\downarrow \bar H=L(1,0,0)\otimes L(1,0,0)^{[1]}+L(0,0,1)\otimes L(0,0,1)^{[1]}$.\footnote{This statement is made without loss of generality: one can embed with graph automorphisms to have dual versions of these modules.} Now \ref{h1fora1a2b2a3} implies $H^1(\bar H,Q)=0$.

In case $(B_4,E_6,2)$ we must have $\bar H\leq D_5$, with $Q$ a spin module for $L'$. But then $Q\downarrow \bar H=V\cong L(0001)$ using \cite[2.7]{LS96} is a spin module for $H$ with $V(0001)=L(0001)$. So $H^1(B_4,V)=0$ and this case is ruled out. 

Lastly take case $(X,G,p)=(C_3,\bullet,2)$ of type $C_3$. We need an $L'$-ir embedding of $\bar H$ in $L'$ and an $H$-composition factor $V$ of $Q$ with $H^1(H,V)\neq 0$. We will see this is impossible. As above, if $G=E_6$, $L'$ has to be type $A_5$, with $Q$ having $L'$-composition factors $k$ and $L(0,0,1,0,0)=\bigwedge^3L(1,0,0,0,0)$ Hence $Q$ has $\bar H$ composition factors which are $k$ or in $\bigwedge^3L(1,0,0)$ which has composition factors $L(0,0,1)|L(1,0,0)^2$. Since these do not appear in \ref{h1fora1a2b2a3} this case is ruled out. Similarly if $G=E_7$ or $E_8$ we must still have $L'=A_5$ and we must also consider the restrictions of $L(0,1,0,0,0)$ and its dual, $L(0,0,0,1,0)$ to $\bar H$. These are $\bigwedge^2L(1,0,0)\cong \bigwedge^4L(1,0,0)$ which also contain no composition factors with non-trivial $H^1$.

Since there are no embeddings of a subgroup of type $C_4$ into any proper Levi of $E_6$, this case is ruled out too.

This completes the proof of the first statement of Theorem 1.

\subsection*{$p\in N(X,G)$ implies the existence of a non-$G$-cr subgroup $H$ with root system $X$}

The examples when $G=G_2$ and $F_4$ were shown already in \cite[Theorem 1]{SteG2} and \cite[Theorem 1(A)(B)]{SteF4} to be non-$G$-cr, so we need only deal with the cases $G=E_6,\ E_7$ and $E_8$.

{\it Proof of the second part of Theorem 1: The subgroups listed in Table \ref{t1} are non-$G$-cr:}

The proof of many of these cases is similar. Let $H=E(X,G,p)$ for one of the examples in Table \ref{t1}. We locate $H$ within a parabolic subgroup of $G$ and establish the embedding $\bar H\leq L$. Next we take a low dimensional (faithful) $G$-module $V$ and calculate the restriction to $H$ and $\bar H$ of this $G$-module; in all cases under consideration these will be non-isomorphic. Thus we can conclude that since $V\downarrow H\not\cong V\downarrow \bar H$, $H$ is not even $GL(V)$-conjugate to $\bar H$, let alone $G$-conjugate to $\bar H$. Further, in all the cases under consideration we will conclude that if $H$ is non-$F_4$-cr, then it is also non-$E_r$-cr for $6\leq r\leq 8$ using the embeddings $F_4\leq E_6\leq E_7\leq E_8$; unfortunately we seem to need to do this mostly case by case. 

We will now give a few examples. 

\underline{$H=E(E_6,A_1,2)$}

Here $H$ is a subgroup of type $A_1$ in a subsystem $A_1^2$ given by $A_1\hookrightarrow A_1^2$ by $x\mapsto (x,x)$. From \cite[Table 10.1]{LS04} we have $V_{27}\downarrow F_4=V_{26}+k$. Now from \cite[5.1]{SteF4} we have $V_{26}\downarrow A_1^2=L(1,1)+k^6+L(1,0)^4+L(0,1)^4$, so $V_{27}\downarrow H=L(1)\otimes L(1)+L(1)^8+k^7=T(2)+L(1)^8+k^7$. In \cite{SteF4} it is shown that $H$ is in a long $A_1$-parabolic of $F_4$, hence $H$ is in an $A_1$-parabolic of $E_6$ (so that $L'$ of type $A_1$). But $V_{27}\downarrow L'=L(1)^9+k^9$ and so $H$ is not $GL(V_{27})$-conjugate to (a subgroup of) $L'$, let alone $E_6$-conjugate. Now $V_{27}\downarrow L'\cong {V_{27}}^*\downarrow L'$ and $V_{27}\downarrow H\cong {V_{27}}^*\downarrow H$. Since the $E_7$-module $V_{56}$ has $V_{56}\downarrow E_6=V_{27}\oplus {V_{27}}^*+k^2$ we see $H$ is also non-$E_7$-cr.

To show it is also non-$E_8$-cr, note that $L(E_8)\downarrow E_7=L(E_7)+L(T_1)+L(Q)^2$ where $Q$ is the unipotent radical of an $E_7$-parabolic of $E_8$, with $L(Q)\downarrow E_7=V_{56}+k$. Thus $L(E_8)\downarrow H$ contains at least two submodules isomorphic to $T(2)$ (contained in the two $V_{56}$s). On the other hand $L(E_8)\downarrow L'=L(A_1T_1)+k^6+M^2$ where $M$ is the restriction to $L'$ of the Lie algebra of the unipotent radical of an $A_1$-parabolic. Using \ref{abs}, $M$ has composition factors with high weights $1$ or $0$, which must be semisimple since $\Ext^1_{A_1}(L(1),L(1))=\Ext^1_{A_1}(L(1),L(0))=0$. In particular, while the direct summand $L(A_1T_1)$ is an indecomposable module $T(2)$ for $L'$, it is the only one in $L(E_8)$; for $H$ there are at least two such (in $L(Q)$). Thus $H$ is also non-$E_8$-cr.

\underline{$(G,X,p)=(E_6,A_2,3)$}

Let $\tau$ denote a graph automorphism of $G$ with induced action on the Dynkin diagram for $G$. If $G_\tau$ denotes the fixed points of $\tau$ in $G$, we have $G_\tau\cong F_4$ such that the root groups corresponding to simple short roots are contained in the subsystem (of type $A_2A_2$) determined by the nodes in the Dynkin diagram of $G$ on which $\tau$ acts non trivially. Thus $H$ is contained in $A_2\tilde A_2\leq F_4$ by $x\mapsto (x,x)$. It is shown in \cite[4.4.1, 4.4.2]{SteF4} that this subgroup is in a $B_3$-parabolic of $F_4$ with $V_7\downarrow \bar H=L_{A_2}(11)$.

In \cite[5.1]{SteF4} the restrictions of the $F_4$-module $V_{26}=V(0001)\cong 0001/0000$ to $H$ and $\bar H$ is calculated. Using this together with $V_{27}\downarrow F_4=T(0001)=0000/0001/0000$ we see that $V_{27}\downarrow\bar H$ cannot be the same as $V_{27}\downarrow H$: the former is an extension by the trivial module of $V_{26}\downarrow \bar H=11^3+00^5$ where the resulting module is self-dual, so must be $11^3+00^6$ whereas the latter is forced to be $T(11)^3$. By a similar argument as before, we also get that this subgroup is non-$E_7$-cr and non-$E_8$-cr.

We give an example of a subgroup not arising from a non-$F_4$-cr subgroup (these being found in \cite{SteF4}): 

\underline{$(G,X,p)=(E_6,A_1,5)$}. 

The module $T(8)=L(0)/L(8)/L(0)$ is a direct summand of the $25$-dimensional module $L(4)\otimes L(4)=T(4)\otimes T(4)$ by \ref{tilting}. The two tensor factors here admit orthogonal forms, so the tensor product does too. Hence we get a subgroup of type $SL_2$ in $GL_{25}$ which is actually contained in $SO_{25}$. Indeed as the $10$-dimensional direct factor $T(8)$ is the unique such, the duality must preserve this factor. Hence we get an $A_1\leq SO_{10}\times SO_{15}$ and so projecting to the first orthogonal group, we get $H\leq SO_{10}$ with $V_{10}|H=T(8)$. 

Now, by \ref{girclassical} we have that this subgroup is in a parabolic of $SO_{10}$. Considering dimensions of composition factors of Levi subgroups of $D_5$ acting on the natural module shows that $H$ must in fact be in a $D_4$-parabolic of $D_5$ with $\bar H$ being $D_4$-irreducible and $V_8\downarrow \bar H=L(8)$. By e.g. \cite[5.1]{SteF4} we can calculate $V_{27}\downarrow D_4=L(\omega_1)+L(\omega_3)+L(\omega_4)+k^3$. We wish to restrict this further to get $V_{27}|\bar H$ and $V_{27}|H$. Note that since $L(8)\cong L(3)\otimes L(1)^{[1]}$, we have $\bar H\leq Sp_4\times Sp_2\leq D_4$. Let $\bar H'$ (resp. $\bar H''$) denote the projection of the $\bar H$ in the first (resp. second) factor. 
Taking a graph automorphism, we can consider $SL_4$ as type $D_3$ corresponding to nodes $2$, $3$ and $4$ of the Dynkin diagram. Then we have $L_{D_4}(\omega_1)| SL_4=L(010)+k^2$, thus $L_{D_4}(\omega_1)|\bar H'=\bigwedge^2(L(3))+k^2=L(4)+k^3$, with $L_{D_4}(\omega_1)|\bar H=L(4)+L(1)^{[1]}+k$ or $L_{D_4}(\omega_1)|\bar H=L(4)+L(2)^{[1]}$. As $\bar H$ is $D_4$-ir, it must be the latter, since $L(1)^{[1]}$ carries a symplectic form. Also from \cite[2.7]{LS96} one sees that $L_{D_4}(\omega_3)\downarrow SL_4\cong L_{D_4}(\omega_4)\downarrow SL_4=L(100)+L(001)$ and so $L_{D_4}(\omega_3)\cong L_{D_4}(\omega_4)| \bar H'=L(3)^2$. Thus $L_{D_4}(\omega_3)\cong L_{D_4}(\omega_4)|\bar H=L(3)\otimes L(1)^{[1]}=L(8)$.

Finally we conclude that $V_{27}\downarrow \bar H=L(8)^2+L(4)+L(2)^{[1]}+0^3$. In particular, $\bar H$ acts semisimply. On the other hand $V_{27}\downarrow D_5=L(\omega_1)+L(\omega_5)+k$ where $L(\omega_5)$. But $H$ does not act semisimply on $V_{10}$. So $\bar H$ is not $GL(V_{27})$-conjugate to $H$, so neither is it $E_6$-conjugate to $H$.

The remaining cases where $X=A_1$ are similar.

Let us now vouch for the existence of the subgroup asserted in case 

\underline{$(G,X,p)=(E_8,C_3,3)$.} 

First observe that since the natural module $L(100)$ for $Sp_6$ admits a symplectic form, the tensor square $M=L(100)\otimes L(100)$ admits an orthogonal form, with composition factors $L(200)|L(010)|L(000)^2$. Since $L(100)$ is a tilting module, so is $M$; and since $L(200)=V(200)=T(200)$, while \cite{Lub01} gives $V(010)=L(010)/L(000)$ we must have $M\cong L(200)+T(010)$. Duality preserves these factors, so the $15$-dimensional module $T(010)$ is orthogonal for $Sp_6$. Thus we have a subgroup $Sp_6\leq SO_{15}\leq SO_{16}$ obviously in a $D_7$-parabolic of this $D_8$.

\underline{$H=E(C_4,E_7,2)$} 

is discussed in \cite[2.7, Proof]{LST96}; there it is shown to be in an $E_6$-parabolic and not conjugate to its image $\bar H\cong C_4\leq F_4\leq E_6=L'$. We need to show that this subgroup is also non-$E_8$-cr. For this, restriction of $L(E_8)$ to an $E_6$ Levi gives $L(E_8)|E_6\cong L(E_6T_2)+L(Q)+L(Q^-)$, with $L(Q)$ having composition factors $k$, $L(\omega_1)$ or $L(\omega_6)$ by \ref{abs}. We have $L(\omega_6)|\bar H=L(\omega_1)|\bar H=L(0100)+k$, and $L(E_6T_2)\cong L(E_6)+k^2$ has dimension $80$. On the other hand, $L(E_8)|E_7=L(E_7T_1)+L(R)+L(R^-)$ for $R$ the unipotent radical of an $E_7$-parabolic. By \ref{abs} $L(R)|E_7=V_{56}+k$. But $V_{56}|A_7=L(\lambda_2)+L(\lambda_6)$ from \cite[2.?]{LS96}. Thus $V_{56}|H=\bigwedge^2(L(1000))+(\bigwedge^2(L(1000)))^*=T(0100)^2$.\footnote{One way to see this is to note that $T=L(1000)\otimes L(1000)$ is a tilting module, whose character can be decomposed to yield composition factors $L(2000)|L(0100)^2|L(0000)^4$. Now, one can use  Doty's Weyl group package for GAP to see that $V_{C_4}(2000)$ is uniserial with successive factors $L(2000)|L(0000)|L(0100)|L(0000)$ and $V_{C_4}(0100)$ is uniserial with successive factors $L(0100)|L(0000)$. Thus $T(2000)$ is uniserial with successive factors $L(0000)|L(0100)|L(0000)|L(2000)|L(0000)|L(0100)|L(0000)$ (it is clear that it has both a Weyl- and dual Weyl-filtration). So $T=T(2000)$ and indecomposable. But $\bigwedge^2(1000)$ is a submodule of $T$; dimension considerations imply that it consists of the last three factors. But $T(0100)=L(0000)|L(0100)|L(0000)$ so the claim follows.} In particular there are $4$ direct factors in $L(E_8)|H$ which are isomorphic to the $28$-dimensional module $T(0100)$. However we found above that there are none in the submodule $(L(Q)+L(Q^-))|\bar H$ of $L(E_8)|\bar H$, so if $H$ were conjugate to $\bar H$, one would have to find these $4$ direct factors $T(0100)$ inside $L(E_6T_1)$; but the dimension of the latter is $79<4\times 28=112$.

There is one further case where we could not give a nice embedding as we have done above. Let 

\underline{$H=(E_7,G_2,7)$.}

We first indicate how to see the existence of this subgroup then show that it cannot have any proper reductive overgroup. By \cite{LS04}, when $p=7$, $F_4$ has a maximal subgroup of type $G_2$. Set $\bar H$ to be this subgroup and regard $\bar H$ as subgroup of a Levi subgroup of an $E_6$-parabolic; note that $\bar H$ is $E_6$-irreducible. By \ref{excepmax} one has $V_{27}|\bar H=L(20)+k$. Now, using \cite{Lub01}, one has, when $p=7$ that $V(20)$ is uniserial with composition factors $20|00$. Thus $H^1(\bar H, L(20))=H^0(\bar H, H^0(20)/L(20))=k$. Now $Q|L'\cong V_{27}$ or ${V_{27}}^*$ so one has $H^1(\bar H,Q)=k$. Now by \cite[3.2.15]{SteF4} it follows that there is a non-$G$-cr subgroup $H$, which is a complement to $Q$ in $\bar HQ$.

Suppose $H$ had a proper reductive overgroup in $G$. Then by \ref{excepmax} it would have to lie in a subsystem subgroup of type $A_7$. Also it cannot lie in any parabolic subgroup of $A_7$ since then $\bar H$ would not be $E_6$-irreducible. Checking \cite{Lub01} one sees that there are no irreducible $8$-dimensional representations of $H\cong G_2$. This is a contradiction. Thus $H$ has no proper reductive overgroup in $G$ as required.

The remaining cases are all similar and easier. This completes the proof of Theorem 1.


%
%
\subsection*{Acknowledgements} We would like to thank Martin Liebeck for helpful comments and corrections on a previous version of the paper.
{\footnotesize
\bibliographystyle{amsalpha}
\bibliography{/Users/dis20/Documents/Maths/Work/stewart}}

\providecommand{\bysame}{\leavevmode\hbox to3em{\hrulefill}\thinspace}
\providecommand{\MR}{\relax\ifhmode\unskip\space\fi MR }
\providecommand{\MRhref}[2]{%
  \href{http://www.ams.org/mathscinet-getitem?mr=#1}{#2}
}
\providecommand{\href}[2]{#2}
\begin{thebibliography}{CPSvdK77}

\bibitem[ABS90]{ABS90}
H.~Azad, M.~Barry, and G.~Seitz, \emph{On the structure of parabolic
  subgroups}, Comm. Algebra \textbf{18} (1990), no.~2, 551--562. \MR{MR1047327
  (91d:20048)}

\bibitem[BMR05]{BMR05}
Michael Bate, Benjamin Martin, and Gerhard R{{\"o}}hrle, \emph{A geometric
  approach to complete reducibility}, Invent. Math. \textbf{161} (2005), no.~1,
  177--218. \MR{MR2178661 (2007k:20101)}

\bibitem[Bou82]{Bourb82}
Nicolas Bourbaki, \emph{\'{E}l{\'e}ments de math{\'e}matique: groupes et
  alg{\`e}bres de {L}ie}, Masson, Paris, 1982, Chapitre 9. Groupes de Lie
  r{{\'e}}els compacts. [Chapter 9. Compact real Lie groups]. \MR{682756
  (84i:22001)}

\bibitem[Cli79]{Cli79}
Ed~Cline, \emph{{${\rm Ext}^{1}$} for {${\rm SL}_{2}$}}, Comm. Algebra
  \textbf{7} (1979), no.~1, 107--111. \MR{514867 (80e:20054)}

\bibitem[CPSvdK77]{CPSV77}
E.~Cline, B.~Parshall, L.~Scott, and W.~van~der Kallen, \emph{Rational and
  generic cohomology}, Invent. Math. \textbf{39} (1977), no.~2, 143--163.
  \MR{MR0439856 (55 \#12737)}

\bibitem[Don93]{Don93}
Stephen Donkin, \emph{On tilting modules for algebraic groups}, Math. Z.
  \textbf{212} (1993), no.~1, 39--60. \MR{MR1200163 (94b:20045)}

\bibitem[DS96]{DS96}
Michael~F. Dowd and Peter Sin, \emph{On representations of algebraic groups in
  characteristic two}, Comm. Algebra \textbf{24} (1996), no.~8, 2597--2686.
  \MR{1393277 (97c:20066)}

\bibitem[Jan03]{Jan03}
Jens~Carsten Jantzen, \emph{Representations of algebraic groups}, second ed.,
  Mathematical Surveys and Monographs, vol. 107, American Mathematical Society,
  Providence, RI, 2003. \MR{MR2015057 (2004h:20061)}

\bibitem[LS96]{LS96}
Martin~W. Liebeck and Gary~M. Seitz, \emph{Reductive subgroups of exceptional
  algebraic groups}, Mem. Amer. Math. Soc. \textbf{121} (1996), no.~580,
  vi+111. \MR{MR1329942 (96i:20059)}

\bibitem[LS04]{LS04}
\bysame, \emph{The maximal subgroups of positive dimension in exceptional
  algebraic groups}, Mem. Amer. Math. Soc. \textbf{169} (2004), no.~802,
  vi+227. \MR{MR2044850 (2005b:20082)}

\bibitem[LST96]{LST96}
Martin~W. Liebeck, Jan Saxl, and Donna~M. Testerman, \emph{Simple subgroups of
  large rank in groups of {L}ie type}, Proc. London Math. Soc. (3) \textbf{72}
  (1996), no.~2, 425--457. \MR{1367085 (96k:20087)}

\bibitem[L{\"u}b01]{Lub01}
Frank L{\"u}beck, \emph{Small degree representations of finite {C}hevalley
  groups in defining characteristic}, LMS J. Comput. Math. \textbf{4} (2001),
  135--169 (electronic). \MR{1901354 (2003e:20013)}

\bibitem[LY93]{LY93}
Jia~Chun Liu and Jia~Chen Ye, \emph{Extensions of simple modules for the
  algebraic group of type {$G_2$}}, Comm. Algebra \textbf{21} (1993), no.~6,
  1909--1946. \MR{1215553 (94h:20051)}

\bibitem[McN10]{McN10}
George~J. McNinch, \emph{Levi decompositions of a linear algebraic group},
  Transform. Groups \textbf{15} (2010), no.~4, 937--964. \MR{2753264
  (2012b:20110)}

\bibitem[Ser98]{Ser98}
J.-P. Serre, \emph{1998 {M}oursund lectures at the {U}niversity of {O}regon}.

\bibitem[Sin94]{Sin94}
Peter Sin, \emph{Extensions of simple modules for special algebraic groups}, J.
  Algebra \textbf{170} (1994), no.~3, 1011--1034. \MR{1305273 (95i:20066)}

\bibitem[Ste10a]{SteG2}
David~I. Stewart, \emph{The reductive subgroups of {$G_2$}}, J. Group Theory
  \textbf{13} (2010), no.~1, 117--130. \MR{2604850 (2011c:20099)}

\bibitem[Ste10b]{SteRes}
\bysame, \emph{Restriction maps on 1-cohomology of (algebraic) groups}.

\bibitem[Ste12]{SteF4}
\bysame, \emph{The reductive subgroups of {$F_4$}}, Mem. Amer. Math. Soc.
  (2012), (to appear).

\bibitem[Ye90]{Ye90}
Jia~Chen Ye, \emph{Extensions of simple modules for the group {${\rm
  Sp}(4,K)$}}, J. London Math. Soc. (2) \textbf{41} (1990), no.~1, 51--62.
  \MR{1063542 (91j:20105a)}

\bibitem[Yeh82]{Yeh82}
S.~el~B. Yehia, \emph{Extensions of simple modules for the chevalley groups and
  its parabolic subgroups}, Ph.D. thesis, University of Warwick, 1982.

\end{thebibliography}

\end{document}